\documentclass[12pt,reqno]{article}

\usepackage{amsmath,amsthm,amssymb,amsfonts}
\usepackage{xcolor}
\usepackage{graphicx}
\usepackage{verbatim}
\usepackage[inline]{enumitem}

\usepackage{etoolbox}
\usepackage{authblk}

\usepackage{hyperref}

\theoremstyle{plain}
\newtheorem{theorem}{Theorem}
\newtheorem{corollary}[theorem]{Corollary}
\newtheorem{lemma}[theorem]{Lemma}
\newtheorem{proposition}[theorem]{Proposition}

\theoremstyle{definition}
\newtheorem{definition}{Definition}
\newtheorem{example}{Example}
\newcommand{\setEnvironmentQed}[2]{
  \AtBeginEnvironment{#1}{%
    \pushQED{\qed}\renewcommand{\qedsymbol}{#2}%
  }
  \AtEndEnvironment{#1}{\popQED}
}
\setEnvironmentQed{example}{\ensuremath{\blacksquare}}

\theoremstyle{remark}

\newtheorem*{remark*}{Remark}

\ExplSyntaxOn
\NewDocumentCommand{\iqtest}{m}{
\seq_set_split:Nnn\l_tmpa_seq{,}{#1}
\begin{enumerate*}[label=\textbf{\alph*)},itemjoin=\quad\quad\quad]
\seq_map_inline:Nn \l_tmpa_seq { \item ##1 }
\end{enumerate*}}
\ExplSyntaxOff

\begin{document}

\title{Self-Referential Tests}

\author{Siyona Agarwal}
\author{Julian Bernhoft}
\author{Karam Gill}
\author{Yonis Gulleth}
\author{Eric Huang}
\author{Benjamin Li}
\author{Brandon Ni}
\author{Soham Samanta}
\author{Leone Seidel}
\author{Boya Yun}
\affil{PRIMES STEP}
\author{Tanya Khovanova}
\affil{MIT}

\maketitle

\begin{abstract}
We study self-referential multiple-choice tests with the question: \emph{How many correct answer choices are there?}
The answer choices are positive integers. A value $a$ is called \emph{valid} if it occurs exactly $a$ times among answer choices. The \emph{cost} of a test is the sum of all answer choices, linking the problem to integer partitions.

Using this framework, we define solvable and $k$-solvable tests and derive generating functions that enumerate
them by cost, number of distinct valid values, and number of options. We also investigate extremal questions, including minimum costs and the maximum possible number of valid values. The paper was inspired by a puzzle from \emph{Mathematical Puzzles and Curiosities}.
\end{abstract}

\section{Original Problem}

The following puzzle appeared in the book \textit{Mathematical Puzzles and Curiosities} by Ivo David, Tanya Khovanova, and Yogev Shpilman \cite{DKS2026}. The puzzle is abridged and edited here.

\begin{quote}
While browsing the self-referential shop, you see the following poster:

\begin{center}
\textit{If you choose the answer to this question at random, what is the probability you will be correct?}

    \textbf{a) 25\%	\ \ \ \	b) 50\%	\ \ \ \		c) 60\%	\ \ \ \		d) 25\%}

\textit{Buy an instant multiple-choice IQ test here.}
\end{center}

Multiple-choice IQ tests are printed instantly for your convenience at the shop. The test asks the same question as the poster and lists several answers --- each a positive integer, followed by \%. To buy an IQ test, you must select its cost $C$ and a number of answer options $K$.

The example above has $K = 4$ and $C = 160$. Answers $50$ and $60$ each appear with frequency 25\%, while answer $25$ appears with frequency 50\%. Thus, our sample test does not have a correct answer. Here you can find three more examples of what other tests might look like.
\begin{center}
\begin{tabular}{|l|r|c|l|}
\hline
 & $C$ & $K$ & test printed \\ \hline
Test 1 & \$100 & 5 & a) 10\% b) 20\% c) 20\% d) 25\% e) 25\% \\
Test 2 & \$50 & 4 & a) 5\% b) 10\% c) 10\% d) 25\% \\
Test 3 & \$250 & 4 & a) 25\% b) 75\% c) 75\% d) 75\% \\ 
\hline
\end{tabular}
\end{center}

Question 1. Suppose you wish to purchase a \$100 multiple-choice IQ test with a correct answer. What is the maximum number of answers the test may have? What is the correct answer for that test?

Question 2. What is the cost of the cheapest IQ test with a correct answer?
\end{quote}

\section{Introduction}

Self-referential multiple-choice questions are a classic source of logical amusement. In this paper, we study a family of such questions inspired by a puzzle about customized IQ tests, where the possible answers are constrained by a prescribed total cost. The original version involves percentages and asks for the probability that a randomly chosen answer is correct. This leads naturally to questions about when a test has no correct answer, exactly one correct answer, or several seemingly correct answers.

In Section~\ref{sec:NoPercentages}, we remove the special role of the number $100$. We introduce a simpler and more flexible model, which we call the \emph{natural test}. In this setting, the answer choices are positive integers, and the question asks how many correct answer choices there are. In both formulations, a test is a partition of its cost. However, in the second formulation, a value can be considered correct when its multiplicity equals the value itself.

The original puzzle asks about correct answers; however, the notion of a correct answer is confusing. Thus, in Section~\ref{sec:SolvableandUnsolvableTests}, we introduce solvable and unsolvable tests. For solvable tests, we consider monosolvable, polysolvable, and $k$-solvable tests.

Viewing tests as partitions allows us to apply methods from partition theory and generating functions. In Section~\ref{sec:genfunctions}, we derive generating functions that count solvable tests by cost, number of distinct valid answers, and the number of options. As corollaries, we get the generating functions for unsolvable and monosolvable tests by cost.

In Section~\ref{sec:minandmax}, we investigate extremal questions, such as for which $n$ and $k$ there are no tests with $k$ valid answers and cost $n$. We also derive the smallest cost of a $k$-solvable test and show that the smallest cost of a $k$-solvable test with $K$ options equals $K$ plus a constant for large $K$. We calculate the maximum number of options for a solvable test of a given cost and the value of the valid answer for such a test. We calculate the maximum possible number of distinct valid answers for a given cost or number of options. We also discuss the special class of omni-valid tests, in which every answer choice is valid.

In Section~\ref{sec:perctests}, we go back to percentage tests. We show that the maximum number of valid answers in such a test is $13$. We show that a solvable test exists if and only if the cost is greater than or equal to $19$. We also calculate the minimum cost of a $k$-solvable test and the maximum number of options in a solvable test.

\section{No Percentages}
\label{sec:NoPercentages}

The original puzzle depends on the notion of percentages, where the number $100$ plays a special role. To get rid of this extra parameter, we want to create a new test type. Our answers are now positive integers, and the question becomes 
\begin{quote}
    How many correct answer choices are there?
\end{quote}
The cost of such a test is the sum of all the values. Thus, such a test can be viewed as a \textit{partition} of its cost. We denote the number of partitions of $n$ by $p(n)$.

\begin{example}
Here we have an example of a test of cost 6. How many correct answer choices are there?
\begin{center}
\iqtest{1,1,2,2}
\end{center}
\end{example}

We call the IQ test from the book the \textit{percentage test}, and the new test the \textit{natural test}. We chose the latter name for the following two reasons. The new test is natural as it does not contain an extra parameter of $100$, and it is based on natural numbers.

Neither test depends on the order of the answers, only on their values and multiplicities. Thus, we can assume that the answers are in non-decreasing order.

For both tests, we denote the number of answers by $K$ and the cost by $C$. Consider a test $(a_1,a_2,\dots,a_K)$ with answers $a_1 \le a_2 \le \dots \le a_K$. By definition, we have
\[a_1 + a_2 + \dots + a_K = C.\]
In other words, the answers form a partition of $C$ for both tests.

We denote by $m_{a}$ the number of times the answer $a$ appears in the test. We call it the \textit{multiplicity} of $a$.

For a given natural test, we call an integer value $a$ \textit{valid} if and only if $a$ appears as the answer choice $a$ times. In other words,
\[a = m_{a}.\]

We can restate the notion of a valid answer for natural tests in terms of Young diagrams, where the row lengths correspond to answers. Suppose we split the Young diagram into maximal horizontal rectangles made of rows of the same length. The answer is valid if the corresponding rectangle is a square.

Table~\ref{tab:Examplesnatural} contains examples of natural tests that are analogous to sample percentage tests in the introduction.

\begin{table}[ht!]
    \centering
\begin{tabular}{|l|r|c|r|}
\hline
 & $C$ & $K$ & test \\ \hline
Test 1 & 10 & 5 & (1, 1, 2, 3, 3) \\
Test 2 & 9 & 4 & (1, 2, 3, 3) \\
Test 3 & 10 & 4 & (1, 3, 3, 3) \\ 
\hline
\end{tabular}
    \caption{Examples of natural tests}
    \label{tab:Examplesnatural}
\end{table}

The original puzzle asks about correct answers. However, the book explains that the notion of the correct answer is nontrivial. For example, in Test 3, one might think that the answers $1$ and $3$ should be considered correct. However, then we have $4$ correct answers in total, creating a contradiction. Thus, we introduce solvable and unsolvable tests in the next section.

The original questions in the percentage puzzles can be reformulated in terms of natural tests in the following way. Question 2 asks for the cost of the cheapest test with a correct answer. We can immediately answer this question for natural tests: The cheapest test with a valid answer has one option, $1$, and a cost of $1$.

Question 1 asks for the maximum number of options a solvable test of a given cost might have and for the value of the valid answer in such a case.

\section{Solvable and Unsolvable Tests}
\label{sec:SolvableandUnsolvableTests}

We call a test \textit{unsolvable} if there are no valid answers. We call a test \textit{solvable} if there is at least one valid answer.

We call a test \textit{monosolvable} if there is exactly one valid answer. We call a test \textit{polysolvable} if it has more than one valid answer. In addition, we call a test $k$-\textit{solvable} if it has exactly $k$ distinct valid answers. In particular, monosolvable tests are also $1$-solvable.

\begin{example}
For both percentage and natural tests, we have the following. Test $1$ is unsolvable. Test $2$ is monosolvable. Test $3$ is polysolvable.
\end{example}

What should we call a correct test or a correct answer, for that matter? An unsolvable test cannot be correct.

Consider a polysolvable test with $k > 1$ distinct valid answers $a_1,\dots, a_k$. Suppose that we call all valid answers correct. Then, the total number $t$ of valid choices is
\[t = (a_1+a_2 + \dots+a_k).\]

However, the value $t$ is greater than any of the valid $a_i$, implying that the values $a_1,\dots, a_k$ are incorrect. Thus, we can argue that we should call an answer \textit{correct} if it is the only valid answer of a monosolvable test.

For which costs $C$ do the solvable tests exist?

\begin{theorem}
\label{thm:existence}
A solvable natural test exists for any positive cost except 2.
\end{theorem}

\begin{proof}
If $C = 1$, then there is only one test, (1), and it is solvable. Suppose $C > 1$. Consider a natural test with $2$ answers: $1$ and $C-1$. As long as $C-1 \neq 1$, this test is solvable. We are left with the case of $C = 2$, in which case we have two possible tests: $(1,1)$ and $(2)$, neither of which is solvable.
\end{proof}

\section{Generating Functions}
\label{sec:genfunctions}

If $p(n,k)$ is the number of partitions of $n$ into exactly $k$ summands, then the bivariate generating function $\mathcal{P}(x,y)$ is well known \cite{Andrews, Wilf}:
\[
\mathcal{P}(x,y) = \sum_{n=0}^{\infty}\sum_{k=0}^{\infty} p(n,k)y^k x^n
= \prod_{m=1}^{\infty}\frac{1}{1-yx^m}.
\]

For the number $p(n)$ of partitions of $n$, the ordinary generating function $\mathcal{P}(x)$ is
\[
\mathcal{P}(x) = \sum_{n=0}^{\infty}p(n) x^n
= \prod_{m=1}^{\infty}\frac{1}{1-x^m}.
\]

We also want to count solvable tests. Let $s(n,k,K)$ be the number of tests with cost $n$, exactly $k$ distinct valid answers, and $K$ options. Then the trivariate generating function $\mathcal{S}(x,y,z)$, defined as
\[\mathcal{S}(x, y, z) = \sum_{n \geq 0} \sum_{k \geq 0} \sum_{K \geq 0} s(n,k,K) x^n y^k z^K\]
can be calculated as follows.

\begin{theorem}
\label{thm:trivariatesolvable}
    The trivariate generating function $\mathcal{S}(x,y,z)$ is
    \[\mathcal{S}(x, y, z) = \sum_{n \geq 0} \sum_{k \geq 0} \sum_{K \geq 0} s(n,k,K) x^n y^k z^K = \prod_{i=1}^{\infty} \left(\frac{1}{1-x^iz} + (y-1)x^{i^2}z^i\right).\]
\end{theorem}

\begin{proof}
To construct this generating function, we consider each possible value $i \geq 1$ and determine its contribution based on how many times it appears as an answer in the test. If the value $i$ appears exactly $m$ times:
\begin{itemize}
    \item The contribution to the cost is $im$, giving factor $x^{im}$.
    \item The contribution to the total number of options is $m$, giving factor $z^m$.
    \item If $m = i$, then $i$ is a valid answer, contributing $y$ to the count of valid answers.
    \item If $m \neq i$, then $i$ is not a valid answer, not contributing to the $y$ count.
\end{itemize}

For a fixed value $i$, summing over all possible multiplicities $m \geq 0$, we get
\[\sum_{m=0}^{\infty}\begin{cases}
y x^{i \cdot i} z^{i} & \text{if } m = i \\
x^{i \cdot m} z^m & \text{if } m \neq i
\end{cases} = yx^{i^2}z^i + \sum_{\substack{m \geq 0 \\ m \neq i}} x^{im}z^m.\]

We can simplify the second term:
\[\sum_{\substack{m \geq 0 \\ m \neq i}} x^{im}z^m = \sum_{m=0}^{\infty} x^{im}z^m - x^{i^2}z^i = \frac{1}{1-x^iz} - x^{i^2}z^i.\]

Therefore, the contribution from value $i$ is
\[yx^{i^2}z^i + \frac{1}{1-x^iz} - x^{i^2}z^i = \frac{1}{1-x^iz} + (y-1)x^{i^2}z^i.\]

Taking the product over all values $i \geq 1$, we get
\[\mathcal{S}(x,y,z) = \prod_{i=1}^{\infty} \left(\frac{1}{1-x^iz} + (y-1)x^{i^2}z^i\right).\]
\end{proof}

Let $s(n,k)$ denote the number of partitions of $n$ in which exactly $k$ distinct part sizes $i$ occur with multiplicity equal to $i$. Equivalently, $s(n,k)$ is the number of tests with cost $n$ and exactly $k$ distinct valid answers. We have
\[s(n,k) = \sum_{K=0}^\infty s(n,k,K).\]
The bivariate generating function
\[\mathcal{S}(x,y) = \sum_{n,k\ge 0} s(n,k) x^n y^k\]
can be obtained from Theorem~\ref{thm:trivariatesolvable} by replacing $z$ with $1$ in $\mathcal{S}(x,y,z)$.

\begin{corollary}
\label{cor:bivariatesolvable}
    The bivariate generating function for the number of tests with cost $n$ and exactly $k$ distinct valid answers is
    \[\mathcal{S}(x,y) = \prod_{i=1}^{\infty}
\left( \frac{1}{1-x^i} + (y-1)x^{i^2} \right).\]
\end{corollary}

We now want to calculate the generating function for a fixed $k$. Let
\[
\mathcal{S}_k(x)=\sum_{n\ge0}s(n,k)x^n.
\]

\begin{corollary}
    We have \[
\mathcal{S}_k(x)=
\prod_{j=1}^{\infty}\left(\frac{1}{1-x^j}-x^{j^2}\right)
\sum_{1\le i_1< \cdots <i_k}
\prod_{r=1}^{k}
\frac{x^{i_r^2}(1-x^{i_r})}{1-x^{i_r^2}+x^{i_r^2+i_r}} .
\]
\end{corollary}

\begin{proof}
    From Corollary~\ref{cor:bivariatesolvable}, we can rewrite each factor as
\[
\frac{1}{1-x^i}+(y-1)x^{i^2}
= \left(\frac{1}{1-x^i}-x^{i^2}\right)+y\,x^{i^2}.
\]
Selecting the term $yx^{i^2}$ from exactly $k$ factors and the other term from the remaining factors gives
\[
\mathcal{S}_k(x) =
\sum_{1\le i_1< \cdots <i_k}
x^{i_1^2+ \cdots +i_k^2}
\prod_{j\notin\{i_1,\dots,i_k\}}
\left(\frac{1}{1-x^j}-x^{j^2}\right).
\]

Factoring the product over all $j$ yields
\[
\mathcal{S}_k(x)=
\prod_{j=1}^{\infty}\left(\frac{1}{1-x^j}-x^{j^2}\right)
\sum_{1\le i_1< \cdots <i_k}
\prod_{r=1}^{k}
\frac{x^{i_r^2}(1-x^{i_r})}{1-x^{i_r^2}+x^{i_r^2+i_r}} .
\]
\end{proof}

Table~\ref{tab:s(n,k)} shows values $s(n,k)$. The value of $k$ ranges from $0$ to $3$ and corresponds to rows, and the value of $n$ ranges from $0$ to $18$ and corresponds to columns. The cells with zeros are left blank for better visualization. The table is represented in the OEIS database \cite{OEIS} as an irregular table read by columns.
\begin{table}[ht!]
    \centering
\begin{tabular}{c|cccccccccccccccc}
$k\backslash n$ & 0 & 1 & 2 & 3 & 4 & 5 & 6 & 7 & 8 & 9 & 10 & 11 & 12 & 13 & 14 & 15 \\ \hline
0 & 1 & & 2 & 2 & 3 & 5 & 8 & 9 & 16 & 19 & 29 & 36 & 53 & 65 & 92 & 115 \\
1 & & 1 & & 1 & 2 & 1 & 3 & 6 & 5 & 10 & 11 & 18 & 21 & 32 & 38 & 54 \\
2 & & & & & & 1 & & & 1 & 1 & 2 & 2 & 3 & 4 & 4 & 7 \\
3 & & & & & & & & & & & & & & & 1 &
\end{tabular}
    \caption{Entry in row $k$ and column $n$ is $s(n,k)$}
    \label{tab:s(n,k)}
\end{table}

For example, the row corresponding to $k=1$ describes the number of $1$-solvable tests.
\begin{example}
    The number of 1-solvable tests has a generating function 
    \[\mathcal{S}_1(x) = \sum_{j = 1}^\infty
x^{j^2}
 \prod_{i\ne j} \left(
\frac{1}{1-x^i} - x^{i^2} \right).
\]
The sequence, which is now sequence A391611, enumerates the number of $1$-solvable tests. Starting from index $0$, it is
\[0,\ 1,\ 0,\ 1,\ 2,\ 1,\ 3,\ 6,\ 5,\ 10,\ 11,\ 18,\ 21,\ 32,\ 38,\ \ldots.\]
\end{example}

The generating function $\mathcal{S}_0(x)$ describes the number of unsolvable tests, which is the first row in Table~\ref{tab:s(n,k)} and is also represented by sequence A276429 in the OEIS database \cite{OEIS}. We have
    \[\mathcal{S}_0(x) = \prod_{j=1}^{\infty}\left(\frac{1}{1-x^j}-x^{j^2}\right).\]


We are also interested in the total number of solvable tests. We denote the number of solvable tests with cost $n$ as $s_{>0}(n)$, and the corresponding generating function as $\mathcal{S}_{>0}(x)$. The generating function for such a sequence is the difference between the total number of partitions and the number of unsolvable tests:
\[\mathcal{S}_{>0}(x) = \prod_{k = 1}^\infty \frac{1}{1 - x^k} - \prod_{k = 1}^\infty \frac{1 - x^{k^2} + x^{k^2 + k}}{1 - x^{k}}.\]

Equivalently, this is the number of partitions into parts such that there is a part $p_i$ with multiplicity $p_i$. The sequence starts at index $0$, and this is now a new sequence A392745 in the OEIS:
\[0,\ 1,\ 0,\ 1,\ 2,\ 2,\ 3,\ 6,\ 6,\ 11,\ 13,\ 20,\ 24,\ 36,\ 43,\ 61,\ 77,\ 102,\ 128,\ \ldots.\]

\begin{example}
    We have three solvable tests with cost 6: $(1,2,3)$, $(1,5)$, and $(1,1,2,2)$. The number of partitions of $6$ is $11$. Thus, the probability that a random test is solvable is $\frac{3}{11}$.
\end{example}

\section{Minima and Maxima}
\label{sec:minandmax}

\subsection{No tests with \texorpdfstring{$k$}{k} valid answers}

Denote by $Q_k$ the sum of the first $k$ squares. That is, $Q_k = \frac{k(k+1)(2k+1)}{6}$. 

The following theorem describes when $s(n,k) = 0$, depending on its proximity to $Q_k$.

\begin{theorem}
\label{thm:pnkis0}
    For $k > 0$, we have $s(n,k) = 0$ if and only if
    \[n<\dfrac{k(k+1)(2k+1)}{6}+k+1 = Q_k+k+1,\]
    and 
    \[n \ne \dfrac{k(k+1)(2k+1)}{6} = Q_k.\]
\end{theorem}

\begin{proof}
The smallest possible cost for a test with $k$ valid answers is $Q_k$, where the valid answers are numbers $1$ through $k$. Therefore, for $n < Q_k$ we have $s(n,k) = 0$. 

Also, for $n= Q_k$, we have $s(n,k) = 1$.

Consider a test with cost $n$, such that $Q_k < n < Q_{k} + k + 1$. If such a test has $k$ valid answers and one of them is greater than $k$, the smallest possible cost happens when the valid answers are 1, 2, $\ldots$, $k-2$, $k-1$, $k+1$, which is at least $Q_k + 2k+1$, exceeding our given cost. So we must have the numbers from the set $1, 2, \dots, k$ as our valid answers. They all sum to $Q_k$. To achieve our cost, we must include additional invalid answers, each with a value at least $k+1$, so as not to destroy the validity of the given valid answers. The total will exceed our given cost.

For any $n\ge Q_k +k+1$, we know that $s(n,k) \ge 1$. We can describe one such test with valid answers $1$ through $k$, and one invalid answer $n-Q_k$. Notice that $n-Q_k$ is greater than or equal to $k+1$, making it invalid.
\end{proof}

\subsection{Cheapest \texorpdfstring{$k$}{k}-solvable test}

We start by discussing the minimum cost given $k$, the number of valid answers. In other words, given $k$, we are looking for the smallest nonzero $s(n,k)$. We get the following corollary from Theorem~\ref{thm:pnkis0}.

\begin{corollary}
    The minimum cost of a $k$-solvable natural test is
    \[Q_k = \frac{k(k+1)(2k+1)}{6}.\]
\end{corollary}

\begin{example}
    A $3$-solvable test has a minimum cost of $14$ and has options $1$, $2$, $2$, $3$, $3$, and $3$.
\end{example}

It follows that the maximum number of distinct valid answers for a natural test of cost $C$ is approximately $\sqrt[3]{3C}$.

We now seek the cheapest $k$-solvable natural test with exactly $K$ options. We start with an example for $k = 1$.

\begin{example}
    Consider a $1$-solvable test. If there is one option, the minimum cost of such a test is $1$. If there are two options, the minimum cost is $3$ with options $1$ and $2$. For $K > 3$, the minimum cost of a solvable natural test with exactly $K$ options is $K+2$: a test with $2$ occurrences of $2$ and $K-2$ occurrences of $1$ is solvable and costs $K+2$. The previous pattern does not extend to $K=3$ as the test $(1,2,2)$ is $2$-solvable. The cheapest test in this case is $(1,2,3)$ with a cost of $6$.
\end{example}

The pattern that we see in the example above for $k=1$ can be extended to other $k$.

\begin{theorem}
\label{thm:cheapestksolvabletestwithkoptions}
    For $K > \frac{(k+1)(k+2)}{2}$, the cost of the cheapest $k$-solvable test with $K$ options is $K+\frac{k(k+1)(k+2)}{3}$.
\end{theorem}

\begin{proof}
Suppose that we have $k$ valid answers. Note that it is optimal to always set the remaining values to the same smallest available number. Suppose the remaining answers are equal to $c$. It is also optimal to minimize the set of valid answers. Thus, the set of all the distinct answer values is $\{1, 2, 3, \dots, k+1\}$.

The only $(k+1)$-solvable test with valid answers $\{1, 2, 3, \dots, k+1\}$ has to have 
\[\sum_{a=1}^{k+1}a=\frac{(k+1)(k+2)}{2}\]
options, and the cost of such a test is 
\[\sum_{a = 1}^{k + 1}a^2 = \frac{(k + 1)(k + 2)(2k + 3)}{6}.\]

By our assumption, our test has more options, implying that if its set of answers is $\{1, 2, 3, \dots, k+1\}$, not all of the answers are valid and the test cannot be $(k+1)$-solvable. Moreover, such a test is the cheapest if the invalid answer is 1.

The cost of such a test is 
\[\frac{(k+1)(k+2)(2k+3)}{6}+\left(K-\frac{(k+1)(k+2)}{2}\right),\]
which can be simplified to
\[K+\frac{(k + 1)(k + 2)(2k + 3) - 3(k+1)(k+2)}{6} = K + \frac{k(k + 1)(k + 2)}{3}\]
as desired.
\end{proof}

We may observe that this formula holds for $k = 1$. We now apply it to $k = 2$.

\begin{example}
For $k=2$, the theorem above tells us that for $K > 6$, the minimum cost of a $2$-solvable natural test with $K$ options is $K + 8$. 

Now we compute the cost for $K \le 6$. We do this by exhaustive search. If $K = 3$, the smallest cost of $5$ is achieved for the test $(1,2,2)$. If $K = 4$, the smallest cost of $8$ is achieved for the test $(1,2,2,3)$. If $K = 5$, the smallest cost of $11$ is achieved for the test $(1,2,2,3,3)$. If $K = 6$, the smallest cost of $15$ is achieved for the test $(1,2,2,3,3,4)$. We obtain the following sequence for the smallest cost of a 2-solvable test as a function of $K$, starting with $K = 3$: 
\[5,\ 8,\ 11,\ 15,\ 15,\ 16,\ 17,\ \ldots.\]
\end{example}

\subsection{Question 1}

Question 1 in the introduction asks about the maximum number of options for a solvable test of a given cost and the valid answer in such a test. 

A related question asks for the cheapest solvable test with $K$ options. For $k=1$, Theorem~\ref{thm:cheapestksolvabletestwithkoptions} tells us that for $K > 3$, the minimum cost of a $1$-solvable natural test with $K$ options is $K + 2$. 

\begin{theorem}
     For $C = 0$ or $C = 2$, a solvable test does not exist. If $M_C$ is the maximum number of options for a solvable test of cost $C$, we have $M_C=C-2$ for $C > 3$. In addition, $M_1 = 1$, and $M_3 = 2$.
     
     In addition, for $C = 4$ and $C > 5$, such a test has a single valid answer: $2$. For $C=5$, the valid answers in such a test are $1$ and $2$; for $C = 1$ or $C=3$, such a test has a single valid answer: $1$.
\end{theorem}

\begin{proof}
    Consider a solvable test of cost $C$. Suppose one of the valid answers is $b \ne 1$. To increase the number of options, we can split all the remaining answers into ones. After that, if $b > 2$, replacing $b$ with two options of two and splitting the rest of the cost into ones maximizes the number of options. In the end, we get a test with $C-2$ options. This procedure works on tests of cost at least $4$. As part of this procedure, for cost $C = 5$, we get two valid answers: $1$ and $2$; for other costs $C \ge 4$, the only valid answer is $2$.

    Suppose the valid answer is $1$. To maximize the number of options, we can replace the leftover answers with twos if the total cost is odd, and with twos and one three if the total cost is even. The number of options in such a test is $\left\lceil \frac{C}{2} \right\rceil$. For $C \ge 4$, the number of options is smaller than in the first case, with the valid answer $2$. We can check the cases of $C < 4$ manually, completing the proof.
\end{proof}

\subsection{Maximum number of valid answers per cost}

Given the cost, what is the maximum number of distinct valid answers? We can resolve this question explicitly.

\begin{theorem}
The maximum number of distinct valid answers is $k$ for the following costs.
\begin{itemize}
    \item $C = Q_k$;
    \item $Q_k + k < C < Q_{k+1}$;
    \item $Q_{k+1} < C \le Q_{k+1} + k+1$.
\end{itemize} 
\end{theorem}

\begin{proof}
Note that the smallest cost with $k$ distinct valid answers is $Q_k$, as, in this case, we choose the smallest possible values for valid answers.

If $C = Q_k$, we can have $k$ distinct valid answers $1, 2, \dots, k$ and no other answers.

If $Q_k + k < C < Q_{k+1}$, we can have the valid answers $1, 2, \dots, k$ and one invalid answer $C - Q_k > k$. This again yields $k$ valid answers, and we cannot have more valid answers.

If $Q_{k+1} < C \le Q_{k+1} + k+1$, let us assume that it is possible to achieve at least $k+1$ valid answers. If we choose any valid answer greater than $k+1$, the smallest cost we get is when the answers are $1$, $2$, $\ldots$, $k$, $k+2$, which has the cost of $Q_{k+1} + (k+2)^2 - (k+1)^2 = Q_{k+1} + 2k + 3$, exceeding our given cost. So we must have the numbers $1, 2, \dots, k+1$ as our valid answers. The total contribution to the cost is $Q_{k+1}$. But in order to achieve our cost $C$, we must have one or more additional invalid answers which sum to $C-Q_{k+1}$. These invalid answers must also be greater than $0$ and less than or equal to $k+1$. But we already have all numbers less than or equal to $k+1$ as valid answers, so if we add another one of those, we will change the number of times that answer appears, making it invalid. Thus, we cannot have $k+1$ valid answers.

Now, we build a test with exactly $k$ valid answers. We start with a test with answers 1 through $k$, each answer $j$ occurring $j$ times, and add one answer of value $C-Q_{k} > k$. The cost of such a test is $C$, with answer $C-Q_k$ invalid and the other $k$ answers valid, concluding the proof.
\end{proof}

We obtain the following sequence, which is now sequence A394247, for the maximum number of distinct valid answers for a cost $C$, starting with $C=0$:
\[0,\ 1,\ 0,\ 1,\ 1,\ 2,\ 1,\ 1,\ 2,\ 2,\ 2,\ 2,\ 2,\ 2,\ 3,\ 2,\ 2,\ 2,\ 3,\ 3,\ 3,\ 3,\ 3,\ \ldots.\]

\subsection{Maximum number of valid answers per number of options}

Given the number of options, what is the maximum number of distinct valid answers? We can resolve this question explicitly.

Denote by $T_n$ the $n$-th triangular number. That is, $T_n = \frac{n(n+1)}{2}$. The answer depends on how close $K$ is to $T_n$.

\begin{proposition}
For a test with $K$ total options, with $T_n \le K < T_{n+1}$, the maximum number of distinct valid answers is $n$. 

\end{proposition}
\begin{proof}
Note that every valid answer $a_i$ appears exactly $a_i$ times. Thus, a test with distinct valid answers $a_1, a_2, a_3,\dots, a_n$ must have at least $a_1+a_2+a_3+\dots + a_n$ total options. To maximize the number of valid answers we can have, we choose the smallest possible valid answers, that is, answers from the set $1, 2, 3, \dots, n$. Then the number of total options must be at least $1+2+3+\dots + n=T_n$.

It is impossible to add another valid answer, as that would yield at least $T_{n+1}$ total options. It is possible to always have exactly $n$ valid answers. If $T_n \le K < T_{n+1}$, we build our test with $n$ valid answers 1 through $n$ and also add $K - T_n$ occurrences of answer $n+1$, which is invalid.
\end{proof}

\subsection{Omni-Valid Tests}

We define a natural test to be \textit{omni-valid} if all of its answers are valid. In other words, for every value $a$ appearing in the test, we have $a=m_a$. Let us call a cost $C$ \textit{omni-solvable} if it is the cost for an omni-valid test.

Then an omni-solvable cost $C$ must be of the form $a_1^2+a_2^2+\dots + a_n^2$ with distinct $a_i$. In other words, cost $C$ must be representable as the sum of distinct squares. It is already well known that the exceptional costs $C$ that are not omni-solvable form a finite set listed in sequence A001422 in the OEIS \cite{OEIS}. The largest number in this sequence is 128. Thus, we can state that for any $C>128$, the cost is omni-solvable. Omni-solvable costs are described by sequence A003995.

The sequence enumerating the number of omni-valid tests of cost $n$ is the same as sequence A033461: Number of partitions of $n$ into distinct squares.

\section{Percentage Tests}
\label{sec:perctests}

\subsection{Preliminary discussion}

Similarly to natural tests, for the percentage test, we call an answer $a_i$ \textit{valid} if and only if $a_i$ appears in $a_i$\% of answers. In other words,
\[a_i = \frac{100m_{a_i}}{K}.\]

\begin{example}
In Test 1 in the Introduction, there are no valid answers. In Test 2, there is one valid answer: 25\%. In Test 3, there are two valid answers: 25\% and 75\%.
\end{example}

Consider some answer $a$ in a natural test with $K$ total options. Then this would correspond to a $\frac{a}{K}$ chance of selecting this answer (if the answer is valid), or $\frac{100a}{K}\%$ chance. Thus, multiplying every element of a natural test by $\frac{100}{K}$ creates an equivalent percentage test if the resulting numbers are integers. Similarly, multiplying each element in the percentage test by $\frac{K}{100}$ yields an equivalent natural test, again, if the resulting numbers are integers.

\begin{example}
A percentage test with options 25\%, 25\%, 50\%, and 75\% is equivalent to a natural test with options $1$, $1$, $2$, and $3$.
\end{example}

These tests are \textit{equivalent} because $a$ will be valid in a natural test if and only if $\frac{100a}{K}$ is valid in the corresponding percentage test. Note that a percentage test with $100$ options is a natural test.

However, the problem with using this equivalency to switch from natural tests to percentage tests is that $\frac{100a}{K}$ may not be an integer. There are a few ways we could address this. Firstly, we could amend the problem statement for percentage tests to allow for non-integer probabilities. Secondly, we can restrict the number of options, $K$, to be a divisor of $100$. But this gives us a rather artificial constraint. 

Note that we also have a similar issue when switching from percentage tests to natural tests. However, valid percentage answers will correspond to integers when multiplying by $\frac{K}{100}$.

Percentage tests are more interesting when all probabilities are multiples of $\frac{100}{K}$, because other probabilities cannot be valid:

\begin{center}
\iqtest{25\%,50\%,75\%,25\%}
\end{center}

Another difference between the two tests concerns changing the number of options. Let us consider the following equivalent tests. 

The natural test with answers $1, 1, 2, 2$ is equivalent to the percentage test with answers $25\%, 25\%, 50\%, 50\%$. The answer $2$, equivalently $50\%$, is valid. If, let us say, we add another answer, $4$, to the natural test, the answer $2$ stays valid in the new natural test $1$, $1$, $2$, $2$, $4$. However, the corresponding percentage test is $20\%, 20\%, 40\%, 40\%, 80\%$: we cannot just add a number; we need to rescale it.

It is thus much easier to add answers without altering the validity in natural tests than in percentage tests.

\subsection{Valid answers}

We have the following validity criterion for a particular valid answer in the percentage test.

\begin{lemma}[Validity Criterion.]
\label{lem:cc}
Let $K$ be the number of choices on a percentage test. The valid answer $N$ must appear in exactly $\frac{NK}{100}$ slots, and therefore, $NK$ must be a multiple of 100.
\end{lemma}

\begin{proof}
Suppose the valid answer $N$ appears in $s$ slots. By definition, the validity of $N$ means that $\frac{N}{100} = \frac{s}{K}$, implying the criterion.
\end{proof}

\begin{proposition}
\label{prop:maximumnumberofvalidanswers13}
The maximum number of valid answers in a percentage test is $13$.
\end{proposition}

\begin{proof}
Suppose there are $14$ valid answers. At a minimum, they cover frequency values from 1\% to 14\%. The total frequencies exceed 100\% of the answer slots, creating a contradiction.

We now show that $13$ is achievable. Consider a percentage test where each value $i\%$, where $1 \leq i \leq 13$, appears exactly $i$ times. Then, we can set the remaining $9$ options equal to $14$\%. Note that there are a total of $\frac{13 \cdot 14}{2} + 9 = 91 + 9 = 100$ answers. Then, each $1 \leq i \leq 13$ appears exactly $\frac{i \cdot 100}{100} = i$ times. So, all $1 \leq i \leq 13$ are the valid answers. Also note that $14$ is not a valid answer since it appears $9$ times instead of the required $14$ times. Thus, $13$ valid answers are achievable, so this is the maximum number of valid answers.
\end{proof}

Similar to natural tests, we call a percentage test \textit{unsolvable} if there are no valid answers. Continuing the similarity, we can define \textit{solvable}, \textit{monosolvable}, and \textit{polysolvable} percentage tests.

\begin{example}
For the percentage test in the original problem: Test $1$ is unsolvable. Test $2$ is monosolvable. Test $3$ is polysolvable.
\end{example}

\subsection{Mitosis}

\begin{example}
Consider the percentage Test 2, with answers 25\%, 10\%, 10\%, 5\%. It is a monosolvable test with valid answer 25\%. Consider a new percentage test, in which we repeat each answer from Test 2 twice: 25\%, 10\%, 10\%, 5\%, 25\%, 10\%, 10\%, 5\%. This is also a monosolvable test with the valid answer $25$\%.
\end{example}

The example above motivates a new definition.

\begin{definition}
    Consider a percentage test with answers $a_1, a_2, \ldots, a_K$. We define its \textit{$b$-polar mitosis} as a new test with $bK$ answers, where each $a_i$ is repeated $b$ times. Sometimes, we call it a \textit{mitosis} if we do not need to specify $b$.
\end{definition}

We can give a related example for natural tests.

\begin{example}
Consider a natural test with options $2$, $2$, $3$, and $3$. It is a monosolvable test with the valid answer $2$. Consider a new test, where we multiply each option by $2$ and repeat twice: $4$, $4$, $4$, $4$, $6$, $6$, $6$, and $6$. This is also a monosolvable test with the valid answer $4$.
\end{example}

\begin{definition}
    Consider a natural test with answers $a_1, a_2, \ldots, a_K$. We define its \textit{$b$-polar mitosis} as a new test with $bK$ answers, where each $a_i$ is repeated $b$ times and is multiplied by $b$. Sometimes, we call it a \textit{mitosis} if we do not need to specify $b$.
\end{definition}

In terms of Young diagrams, a $b$-polar mitosis can be described as replacing each cell in the original diagram with a $b$-by-$b$ square of cells.

\begin{theorem}
Given a natural or percentage test $T$, its mitosis maps valid answers to valid answers and invalid answers to invalid answers.
\end{theorem}
\begin{proof}
    For a percentage test, let $x$ be the value of one of the valid answers that appears $m$ times in test $T$. Then $x$ appears $bm$ times in $b$-polar mitosis test $T'$. Thus, the proportion of values of $x$ stays the same, so $x$ is still a valid answer. Similarly, an invalid answer stays invalid.

    A similar argument works for a natural test.
\end{proof}

\begin{corollary}
    If there exists a percentage test with cost $C$ with a given set of answers, then for any $b \geq 1$, there is a percentage test with cost $Cb$ with the analogous properties. Similarly, if there exists a natural test with cost $C$ with a given set of answers, then for any $b \geq 1$, there is a natural test with cost $Cb^2$ with the analogous properties.
\end{corollary}

\subsection{Existence of solvable tests}

Question 2 in the original puzzle asks for the cheapest solvable percentage test. We expand this question as follows. For which costs $C$ do the solvable percentage tests exist?

\begin{theorem}
\label{thm:percentageexistence}
For percentage tests of cost less than $19$, no solvable test exists. For any cost greater than or equal to $19$, we can have a monosolvable percentage test.
\end{theorem}

\begin{proof}
Assume there exists a solvable percentage test with $K$ options and cost $C \le 18$, where one of the valid values is $N$. Since all answers are positive integers, we have $K\le C$ and $N\le C$, hence $NK\le 18^2=324$.
By Lemma~\ref{lem:cc}, the number $NK$ must be a multiple of $100$, so, given our bounds, we have $NK=100$ and $(N,K)=(10,10)$.
If $10\%$ is valid with $K=10$, it must occur exactly once. All other $9$ answers must be positive and different from $10$, so the cost is at least $10+9\cdot 1=19$,
contradicting $C\le 18$.

For existence of a test when $C\ge 19$ and $C\neq 28$, take $K=10$ and the test
\[10,\ 1,\ 1,\ 1,\ 1,\ 1,\ 1,\ 1,\ 1,\ C-18.\]
Then $10\%$ occurs once and is valid, while all other values are not valid.

For $C=28$, the test \[25,\ 1,\ 1,\ 1.\]
has $K=4$ and $25\%$ is valid.
\end{proof}

Let $s^\%(n,k)$ denote the number of $k$-solvable percentage tests with cost $n$.

We wrote a program to calculate the number of solvable percentage tests with a cost of $C$. We denote this sequence as $s_{>0}^\%(n)$. The number of solvable percentage tests given the cost is now sequence A392924 in the OEIS \cite{OEIS}. We wrote a program to calculate the terms of this sequence. The first nonzero term occurs at index $19$, as we saw in Theorem~\ref{thm:percentageexistence}:
\[1,\ 1,\ 2,\ 3,\ 5,\ 9,\ 13,\ 19,\ 28,\ 40,\ 54,\ 74,\ 98,\ 129,\ 167,\ 215,\ \ldots.\]

\subsection{Minimum Cost of \texorpdfstring{$k$}{k}-solvable Tests}

A cost of $42$ is achievable for a polysolvable test, as we can see in the next example.

\begin{example}
\label{ex:minimumksolvablepercentage}
Consider a percentage test with $20$ options, where $17$ options are $1\%$, 1 option is $5\%$, and two options are $10\%$. The cost of the test is $42$, and options $5\%$ and $10\%$ are both valid.

Consider a percentage test with $25$ options, where $22$ options are $1\%$, 1 option is $4\%$, and two options are $8\%$. The cost of the test is $42$, and options $4\%$ and $8\%$ are both valid.
\end{example}

\begin{proposition}
    The minimum cost for a percentage test that is polysolvable is $42$.   
\end{proposition}

\begin{proof}
Assume $x_1 < x_2$ are both valid answers. Then, $x_1$ needs to appear $\frac{Kx_1}{100}$ times and $x_2$ needs to appear $\frac{Kx_2}{100}$ times, where both $Kx_1$ and $Kx_2$ are multiples of $100$.

Suppose $\gcd(K,100) \leq 10$, then $x_1 \geq 10$. As $x_2$ is different from $x_1$, we have $\frac{Kx_2}{100} \ge 2$, and $x_2 \geq 20$. Moreover, $x_2$ needs to appear at least twice. Therefore, the cost is at least 50, which is not the minimum according to Example~\ref{ex:minimumksolvablepercentage}. If $K = 40$, we cannot make two distinct valid answers keeping the cost below $42$. What remains are cases $K=20$ and $K=25$. We can manually find the minimum cost, which is presented in Example~\ref{ex:minimumksolvablepercentage}.
\end{proof}

We wrote a program to find the minimum cost for a $k$-solvable percentage test. 

The minimizing $k$-solvable tests for $k=3$ through $k=6$ have $50$ options, with valid answers being the even percentages from $2$ through $2k$, and the leftover answers being $1$.


The minimizing $k$-solvable tests for $k=7$ through $k=12$ have $100$ options with valid answers $2$ through $k+1$ percentages, and the leftover answers are $1$. For $k = 13$, the minimizing test is described in Proposition~\ref{prop:maximumnumberofvalidanswers13} and has valid answers $1$ through $13$, with the leftover options equal to $14$.

The resulting sequence is now sequence A396074 in the OEIS \cite{OEIS}:
\[19,\ 42,\ 72,\ 100,\ 145,\ 211,\ 268,\ 340,\ 430,\ 540,\ 672,\ 828,\ 945.\]


\subsection{Question 1}

Let us return to Question 1 and its generalizations: \textit{Suppose you wish to purchase a multiple-choice IQ test with cost $C$ and a correct answer. What is the maximum number of answers the test may have? What is the correct answer for that test?}

To maximize the total number of answers for a fixed total cost of $C$, we must minimize the average cost of each individual answer. Intuitively, this means that we should have the smallest possible valid answer, and other answers equal to $1$.

\begin{lemma}
For a test with cost $C$ and a valid answer, the maximality of the number of options $K$ implies that the valid answer is not $1$.
\end{lemma}

\begin{proof}
If $1$ is valid, then $m_1=\frac{K}{100}$, so $K$ must be a multiple of $100$. Let $K = 100b$. We have $m_1 = b$, and the remaining $99b$ options must be at least $2$, so the cost is at least $b + 198b = 199b$.

Now we construct a test with the same cost, where $2$ is a valid answer with $150b$ options. We use $3b$ options for answer $2$, we are left with at least $193b$ cost to distribute among $147b$ options. We can do this by assigning a cost of $1$ to most of the options, and the rest to guarantee that we do not introduce more valid answers. As we build a test with more options and the same cost, the valid answer $2$ always produces more options than the valid answer $1$.
\end{proof}

From now on, we assume that in the solvable test maximizing the number of options, there exists a valid answer that is not $1$.

\begin{lemma}
The minimum cost of a percentage test with a valid answer $2\le a\le 100$ is
\[
\frac{a^2-a+100}{\gcd(a,100)}.
\]
\end{lemma}

\begin{proof}
Let $g=\gcd(a,100)$. If $a$ is valid, then
\[
m_a=\frac{aK}{100},
\]
so $K$ must be a multiple of $\frac{100}{g}$. Thus the smallest possible number of options is $\frac{100}{g}$, and then $a$ appears $\frac{a}{g}$ times.

The cheapest such test is obtained by making all other answers equal to $1$. Hence, the cost is at least
\[
a\cdot\frac{a}{g}+
\left(\frac{100}{g}-\frac{a}{g}\right)
=
\frac{a^2-a+100}{g}.
\]
\end{proof}

\begin{example}
If $a=2$, then $\gcd(2,100)=2$, so the minimum cost is
\[
\frac{2^2-2+100}{2}=51.
\]
This is achieved by a test with $50$ options: one answer equal to $2$ and the remaining $49$ answers equal to $1$.
\end{example}

\begin{lemma}
\label{lem:maxoptions}
The maximum number of options for a solvable percentage test with cost $C$ and valid answer $a\neq 1$ is not greater than
\[
\frac{100C}{a^2-a+100}.
\]
\end{lemma}

\begin{proof}
If $a>1$ is valid, then
\[
m_a=\frac{aK}{100}.
\]
The copies of the valid answer contribute
\[
a m_a=\frac{a^2K}{100}
\]
to the cost. All remaining answers are at least $1$, so
\[
C \ge \frac{a^2K}{100}+\left(K-\frac{aK}{100}\right)
=K\cdot\frac{a^2-a+100}{100}.
\]
Therefore
\[
K\le \frac{100C}{a^2-a+100}.
\]
\end{proof}

\begin{example}
If $a=2$, then the number of choices has to be a multiple of $50$, and the lemma gives
\[
K\le \frac{100C}{102}=\frac{50C}{51}.
\]
\end{example}

\begin{theorem}
For $C\ge 1379$, the maximum number of options in a solvable percentage test of cost $C$ is
\[
\begin{cases}
50\left\lfloor \dfrac{C}{51}\right\rfloor, 
& C\not\equiv 1 \pmod{51},\\[1.2ex]
50\left(\left\lfloor \dfrac{C}{51}\right\rfloor-1\right),
& C\equiv 1 \pmod{51}.
\end{cases}
\]
\end{theorem}

\begin{proof}
Consider $C\not\equiv 1 \pmod{51}$. First, the given bound is attainable. Let
\[
q=\left\lfloor \frac{C}{51}\right\rfloor.
\]
Take a test with $50q$ options, of which $q$ are equal to $2$. Then $2$ appears in exactly $\frac{2}{100}\cdot 50q=q$ positions, so $2$ is valid. The remaining $49q$ options are initially set equal to $1$, giving a total cost of $51q$.

Since $C\not\equiv 1 \pmod{51}$, the excess $C-51q$ is either $0$ or at least $2$. If it is nonzero, add this excess to one of the entries equal to $1$. This way, $2$ remains valid, and no other valid answers are created. Hence, the claimed number of options is achievable.

If $C=51q+1$, then $K=50q$ is impossible. Indeed, after the $q$ copies of $2$, the remaining $49q$ answers must be positive. The only way to distribute the remaining cost is to have one more option of $2$, making $2$ invalid. However, $K=50(q-1)$ is attainable by using $q-1$ copies of $2$, making the other answers equal to $1$, and adding the remaining excess $52$ to one of the $1$'s. Moreover, if the valid answer is $2$, the number of options must be a multiple of $50$. Thus, in this case, the highest achievable number of options with the valid answer $2$ is $K=50(q-1)$.

Now let $a>1$ be any valid answer. For $a\ge3$, Lemma~\ref{lem:maxoptions} gives
\[
K \le \frac{100C}{106}=\frac{50C}{53}.
\]
In case $C\not\equiv 1 \pmod{51}$, our bound gives at least
\[
50\left\lfloor\frac{C}{51}\right\rfloor
\ge \frac{50C}{51}-50
\]
options. This exceeds $\frac{50C}{53}$ for $C\ge1352$.

In case $C\equiv 1 \pmod{51}$, our construction gives at least
\[
50\left(\frac{C-1}{51}-1\right)
=\frac{50(C-52)}{51}
\]
options. This exceeds $\frac{50C}{53}$ for $C\ge1379$. This proves the theorem.
\end{proof}

It remains only to check the finite range $1 \le C<1379$. A direct computation over all possible valid answers $a=1,\ldots,100$ shows that the same maximum holds for $C \ge 257$.

\begin{example}
Consider the test of cost $256$ with $9$ instances of $4$\%, $1$ instance of $5$\%, and $215$ instances of $1$\%. It has $225$ options. Option $4$\% appears with probability $4$\% and is valid. The achievable maximum number of options with valid answer $2$\% is $200$. Thus, for $C = 256$, the valid answer $4$\% corresponds to more options than the valid answer $2$\%.
\end{example}

Our program calculated the maximum number of options for a solvable percentage test of cost $C$. The sequence below starts at $C = 19$ and is not sequence A396792  in the OEIS \cite{OEIS}:
\[10,\ 10,\ 10,\ 10,\ 10,\ 20,\ 20,\ 20,\ 20,\ 25,\ 25,\ 25,\ 25,\ \ldots.\] 

The sequence is non-decreasing except for the drop at $C=52$.

\begin{example}
Consider $C = 51$. We can have the valid answer $2$ with $50$ options, where the remaining options are $1$s. This is the maximum possible number of options. 

Consider $C = 52$. If there are $50$ options, then options that are not equal to $1$\% are either $3$\% or two occurrences of $2$\%. In either case, the test is not solvable. However, we can have a solvable test with $40$ options, two of which are $5$\%, supplying the valid answer, and the remaining $38$ options sum to $42$.
\end{example}

Our code also outputs the valid answer for each cost in the previous sequence. In case of ties, our program outputs the smallest valid answer. Here is this sequence starting from index $19$:
\[10,\ 10,\ 10,\ 10,\ 10,\ 5,\ 5,\ 5,\ 5,\ 4,\ 4,\ 4,\ 4,\ \ldots.\]

Starting from $C \ge 257$, such a valid answer is $2$.

Note that there are many costs for which there are multiple tests that may have different valid answers. In that case, our program returns the smallest valid answer. The smallest cost with multiple possible optimal valid-answer values across different optimal tests is $39$.

\begin{example}
For $C=39$, the maximum possible number of options is $25$. One optimal test has one valid answer $4$\%, and the remaining cost of $35$ can be distributed among the remaining $24$ options. Another optimal test has two options with valid answer $8$\%, and the remaining $23$ options are equal to $1$\%.
\end{example}

Our sequences allow us to answer Question 1.

\begin{corollary}
    For a solvable percentage test of cost \$100, the maximum possible number of options is $80$, and the corresponding valid answer is $5$\%. 
\end{corollary}

\section*{Acknowledgments}

We are grateful to the PRIMES STEP program for giving us the opportunity to conduct this research. After we finished our paper, we used ChatGPT to check for mistakes in English, math, and LaTeX.


\begin{thebibliography}{9}

\bibitem{Andrews} George E.~Andrews, \textit{The Theory of Partitions}, Encyclopedia of Mathematics and its Applications, Vol.~2, Addison-Wesley, Reading, MA, 1976.

\bibitem{DKS2026}
I.~David, T.~Khovanova, and Y.~Shpilman,
\emph{Mathematical Puzzles and Curiosities},
World Scientific, 2026.

\bibitem{OEIS}
N.~J.~A.~Sloane, \emph{The On-Line Encyclopedia of Integer Sequences},
\url{https://oeis.org/}, accessed 2026.

\bibitem{Wilf}
Herbert S.~Wilf, \textit{generatingfunctionology}, 2nd ed., A K Peters, Wellesley, MA, 1994.

\end{thebibliography}
\end{document}